\documentclass[12pt,a4paper]{article}
\usepackage{amsmath}
\usepackage{amssymb}
\usepackage[square,comma]{natbib}
\usepackage{enumitem}
\usepackage[T1]{fontenc}
\usepackage{slantsc}
\bibpunct[, ]{(}{)}{;}{a}{,}{,}

\newtheorem{theorem}{Theorem}

\newtheorem{lemma}[theorem]{Lemma}
\newtheorem{corollary}[theorem]{Corollary}
\newtheorem{fact}[theorem]{Fact}

\newtheorem{definition}[theorem]{Definition}

\newtheorem{observation}[theorem]{Observation}

\newcommand{\sentt}{Sent_{L_T}}
\newcommand{\sentpa}{Sent_{\lpa}}
\newcommand{\lt}{L_T}
\newcommand{\lpa}{L_{PA}}

\newcommand{\tria}{\vartriangleleft}

\newcommand{\trias}{\vartriangleleft^*}
\newcommand{\etrias}{\trianglelefteq^*}
\newcommand{\twk}{T^{WK}}
\newcommand{\tsk}{T^{SK}}
\newcommand{\omck}{\omega_1^{CK}}

\newenvironment{proof}{{\bf Proof.} }{$\Box$ \\}

\begin{document}

\title{Comparing truth theories based on Weak and Strong Kleene logic} 
\author{Cezary Cie\'sli\'nski\\ Institute of Philosophy, University of Warsaw\\ Poland
\\ \small e-mail:c.cieslinski@uw.edu.pl
\\ }
\date{}
\maketitle

\section{General background}
In the last decades various axiomatic theories of truth have been proposed in the literature.\footnote{See \citep{Hal11} for a comprehensive presentation.} One of the main research topics in the area has been that of assessing their strength. Ideally, such an assessment should permit us to \textit{compare} theories, that is, to answer questions of the form `Is a theory $Th_1$ stronger or weaker than $Th_2$?'. However, such comparisons are not always straightforward and in some cases even the measure of strength to be used is not obvious at all.

The simplest comparisons can be carried out in terms of inclusion. Thus, we can say that $Th_2$ is not weaker than $Th_1$ iff $Th_1 \subseteq Th_2$ ($Th_2$ is strictly stronger than $Th_1$ if the inclusion is proper).\footnote{For example, it can be easily verified that $UTB$ is strictly stronger in this sense than $TB$. Both $TB$ and $UTB$ are disquotational truth theories; for the exact definitions, see \cite[p. 53-54]{Hal11}. For the discussion of the strength of disquotational theories, see also \citep{Cie15}.} This, however, is a crude measure, as it does not permit us to compare theories for which no such simple inclusions hold. 

Axiomatic theories of truth discussed in the literature are often built over some arithmetical base theory playing the role of the theory of syntax. Such a perspective will be adopted also in this paper: all the truth theories which we discuss will be obtained by extending Peano arithmetic with some axioms employing the new predicate `$T(x)$' (the truth predicate). In view of this, another useful measure of strength, often discussed in the literature, is provided by comparing sets of \textit{arithmetical} consequences of theories. In effect, when employing it, we compare not truth theories taken as wholes but their arithmetical content only. From this point of view, for example, the classical compositional typed truth theory $CT^-$ turns out to be equally strong as $UTB$, since both theories have exactly the same arithmetical consequences even though their truth axioms are quite different.\footnote{Namely, both of them are conservative over Peano arithmetic. For the definition of $CT^-$ and its conservativity over $PA$, see \citep[p. 107ff]{Ciebook} (the proof presented there is an adaptation of the construction of \cite{EnVi15}).}

A still subtler measure of strength has been proposed by \citet{Fuji10}, who introduced the following notion of relative truth-definability. 

\begin{definition}\label{trdefina} Let $Th_1$ and $Th_2$ be theories in one and the same language $\lt$ containing the predicate `$T(x)$' called `the truth predicate'. We say that $Th_1$ is relatively truth-definable in $Th_2$ (or $Th_2$ defines the truth predicate of $Th_1$) iff there is a formula $\theta(x) \in \lt$ such that for every  $\psi \in \lt$, if $Th_1 \vdash \psi$, then $Th_2 \vdash \psi(\theta(x)/T(x))$.
\end{definition} 
The expression `$\psi(\theta(x)/T(x))$' stands for the result of substituting the formula $\theta$ for all the occurrences of `$T$' in $\psi$.

It is easy to observe that if $Th_2$ defines the truth predicate of $Th_1$, then $Th_2$ is arithmetically at least as strong as $Th_1$. However, truth-definability is a stricter relation and it gives us more than mere comparisons of arithmetical strength. When employing the notion of truth-definability, we take into account not only the arithmetical, but also the truth-theoretic content of theories. Indeed, if $Th_2$ defines the truth predicate of $Th_1$, then $Th_2$ contains the resources permitting us to reproduce the very notion of truth characterised by the axioms of $Th_1$. In effect, we can claim that in such a case $Th_2$ is at least as strong as $Th_1$ not just arithmetically but also \textit{conceptually}. \\

The aim of this paper is to compare the conceptual strength of two axiomatic theories of truth. The first of these is $KF$ - a theory designed to capture Kripke's fixed-point model construction based on the Strong Kleene evaluation schema. The second theory is $WKF$, which closely resembles $KF$, except that it is designed to capture the fixed-point construction based on the Weak Kleene evaluation schema. We assume that both theories are formulated in the language $\lt$, which is obtained from the arithmetical language by adding a new one-place predicate `$T(x)$'.\footnote{For a fuller discussion of the resources of the arithmetical part of $\lt$, we refer the reader to the final paragraphs of Section \ref{state}.} Both theories contain Peano arithmetic ($PA$) together with induction for the full language with the truth predicate.

For the list of axioms of $KF$, the reader is referred to \cite[p. 201]{Hal11}; following Halbach, we will denote them as {\scshape kf1--kf13}. Before defining the set of axioms of $WKF$, we adopt the following abbreviations.
\begin{itemize}
\item $\sentt$ is the set of sentences of $\lt$.
\item Let $\varphi \in \sentt$. Then `$D(\varphi)$' (`$x$ is determined') is the formula: `$T(\varphi) \vee T(\neg \varphi)$'.
\item Let $\varphi(x)$ be a formula of $\lt$ with one variable free. Then `$D(\varphi(x))$' is the formula: `$\forall x~\big[ T(\varphi(x)) \vee T(\neg \varphi(x)) \big]$'.
\item The expression `$D(\varphi, \psi)$' abbreviates `$D(\varphi) \wedge D(\psi)$'.
\end{itemize}
In Weak Kleene logic a compound formula will have a determinate truth value (truth or falsity) only if its constituents are also determined. This insight gives rise to the following axiomatisation of the truth theory $WKF$, based on Weak Kleene logic.

\begin{definition}\label{wkf} Axioms {\scshape wkf1--wkf4}, {\scshape wkf7--wkf8} and {\scshape wkf11--wkf13} of $WKF$ are exactly the same as {\scshape kf1--kf4, kf7--kf8} and {\scshape kf11--kf13}; see \citep[p. 201]{Hal11}. We list below only those axioms of $WKF$ which differ from the corresponding axioms of $KF$.

\begin{description}
\item[{\scshape wkf5}] $\forall \varphi, \psi \in \sentt \big[ T \big( \neg (\varphi \wedge \psi ) \big) \equiv \big( D(\varphi, \psi) \wedge (T(\neg \varphi) \vee T(\neg \psi)) \big) \big]$
\item[{\scshape wkf6}] $\forall \varphi, \psi \in \sentt \big[ T (\varphi \vee \psi ) \equiv \big( D(\varphi, \psi) \wedge (T(\varphi) \vee T(\psi)) \big) \big]$
\item[{\scshape wkf9}] $\forall \varphi(x) \in \lt \forall v \in Var \big[ T (\neg \forall v \varphi(v) ) \equiv \big( D(\varphi(x)) \wedge \exists t T(\neg \varphi(t)) \big) \big]$
\item[{\scshape wkf10}] $\forall \varphi(x) \in \lt \forall v \in Var \big[ T (\exists v \varphi(v) ) \equiv \big( D(\varphi(x)) \wedge \exists t T(\varphi(t)) \big) \big]$
\end{description}
\end{definition}

\section{State-of-the-art and motivations}\label{state}

Theories of truth based on Weak Kleene logic have been investigated in the literature from two angles. One important source is the paper of \cite{CaDa91}, which describes some striking differences between the Strong Kleene and the Weak Kleene evaluation schemata in Kripke's model-theoretic constructions (the paper does not discuss axiomatic truth theories). Another key source is \citep{Fuji10}, where the discussion focuses on axiomatic theories of truth, with several results about $WKF$ being presented. 

Starting from the model-theoretic approach, below we sketch the well-known constructions which $KF$ and $WKF$ have been designed to capture. In what follows $Q \in \{\exists, \forall \}$ and $\circ \in \{\wedge, \vee \}$. The expression `$Q_d$' (`$\circ_d$') stands for the dual quantifier (dual binary connective): $Q_d$ is $\exists$ if $Q$ is the universal quantifier and it is the universal quantifier otherwise (similarly, $\circ_d$ is either $\wedge$ or $\vee$ depending on what $\circ$ is). $Tm^c$ is the set of constant terms. The next two definitions characterise the notions of Strong Kleene and Weak Kleene Kripke's jump.

\begin{definition}[Strong Kleene Jump]\label{skjump} Let $S \subseteq \omega$. We define:
\begin{align}
J^{SK}(S) = &\ \ \{\ulcorner t = s \urcorner : val(t) = val(s) \}\nonumber\\
\cup&\ \ \{\ulcorner t \ne s \urcorner: val(t) \ne val(s) \}\nonumber\\
\cup&\ \ \{\ulcorner T(t) \urcorner: val(t) \in S \}\nonumber\\
\cup&\ \ \{\ulcorner \neg T(t) \urcorner: \neg val(t) \in S \vee \neg Sent_{\lt}(val(t))\}\nonumber\\
\cup&\ \ \{\ulcorner \varphi \circ \psi \urcorner: \varphi \in S \circ \psi \in S \}\nonumber\\
\cup&\ \ \{\ulcorner \neg (\varphi \circ \psi) \urcorner: \neg \varphi \in S \circ_d \neg \psi \in S \}\nonumber\\
\cup&\ \ \{\ulcorner Qv \varphi \urcorner: Qt \in Tm^c (\varphi(t) \in S) \}\nonumber\\
\cup&\ \ \{\ulcorner \neg Qv \varphi \urcorner: Q_d t \in Tm^c (\neg \varphi(t) \in S) \}.\nonumber\
\end{align} 
\end{definition}
For a set $S \subseteq \omega$ and $\varphi \in \sentt$, let `$D(S, \varphi)$' (`$\varphi$ is determined in $S$') be a shorthand for: `$\varphi \in S \vee \neg \varphi \in S$'. For a formula $\varphi(x)$ with at most one free variable let `$D(S, \varphi(x))$' abbreviate `$\forall t \in Tm^c (\varphi(t) \in S \vee \neg \varphi(t) \in S)$'. We will use also the expression `$D(S, \varphi, \psi)$' as a shorthand for `$D(S, \varphi) \wedge D(S, \psi)$'.

\begin{definition}[Weak Kleene Jump]\label{wkjump}

\begin{align}
J^{WK}(S) = &\ \ \{\ulcorner t = s \urcorner : val(t) = val(s) \}\nonumber\\
\cup&\ \ \{\ulcorner t \ne s \urcorner: val(t) \ne val(s) \}\nonumber\\
\cup&\ \ \{\ulcorner T(t) \urcorner: val(t) \in S \}\nonumber\\
\cup&\ \ \{\ulcorner \neg T(t) \urcorner: \neg val(t) \in S \vee \neg Sent_{\lt}(val(t))\}\nonumber\\
\cup&\ \ \{\ulcorner \varphi \circ \psi \urcorner: D(S, \varphi, \psi) \wedge (\varphi \in S \circ \psi \in S) \}\nonumber\\
\cup&\ \ \{\ulcorner \neg (\varphi \circ \psi) \urcorner: D(S, \varphi, \psi) \wedge (\neg \varphi \in S \circ_d \neg \psi \in S) \}\nonumber\\
\cup&\ \ \{\ulcorner Qv \varphi \urcorner: D(S, \varphi(v)) \wedge Qt \in Tm^c (\varphi(t) \in S) \}\nonumber\\
\cup&\ \ \{\ulcorner \neg Qv \varphi \urcorner: D(S, \varphi(v)) \wedge Q_d t \in Tm^c (\neg \varphi(t) \in S) \}.\nonumber\
\end{align} 
\end{definition}
Employing the jump operations, Weak Kleene and Strong Kleene fixed-point models can be now defined in the following way.

\begin{definition}\label{lfp} Let $E$ be either Weak Kleene or Strong Kleene evaluation scheme (that is, `$E$' is either `$WK$' or `$SK$'). We define:
\begin{itemize}
\item $T^E_0 = Th(N) =$ the set of aritmetical sentences true in $N$ (the standard model of arithmetic),
\item $T^E_{\alpha +1} = J^E(T^E_{\alpha})$,
\item $T^E_{\lambda} = \underset{\alpha < \lambda}{\bigcup} T^E_{\alpha}$,
\item $T^E$ is $T^E_{\kappa}$ for the least $\kappa$ such that $T^E_{\kappa} = T^E_{\kappa + 1}$.
\end{itemize}
\end{definition}
In effect, $(N, \tsk)$ and $(N, \twk)$ are the least fixed-point models based on Strong Kleene and Weak Kleene logic. They are also models for the axiomatic theories $KF$ and $WKF$, respectively.

Below we introduce one additional piece of notation. 

\begin{definition}\label{clord} Let $E$ be either Weak Kleene or Strong Kleene evaluation scheme. We define $ClOrd_{(N, T^E )}$ (the closure ordinal of $(N, T^E )$) as the least $\kappa$ such that $T^E_{\kappa} = T^E_{\kappa + 1}$.
\end{definition}

It transpires that, although similar, both constructions differ in some important respects. In particular, closure ordinals of Weak and Strong Kleene model-theoretic constructions can be different. 

\begin{theorem}\label{theclord} \hfil
\begin{itemize}
\item[(a)] $ClOrd_{(N, \tsk )} = \omck$,
\item[(b)] $ClOrd_{(N, \twk )}$ can be $\omega$ or $\omck$ or infinitely many ordinals in between.
\end{itemize}
\end{theorem}

Part (a) is due to \cite{Kri75}; part (b) is due to \cite{CaDa91}. As observed by Cain and Damnjanovic, the closure ordinal of Weak Kleene model-theoretic construction depends both on the choice of coding and on the choice of the arithmetical language (namely, on the availability of the function symbols). To be more exact, if our arithmetical language is that of Peano arithmetic (with only the symbols for addition, multiplication, successor and $0$ available), the closure ordinal will depend on the choice of coding. On the other hand, if our arithmetical language contains a function symbol for every primitive recursive function, then (independently of coding) the closure ordinal will be $\omck$.\footnote{See also \citep{Spe17}, where it is demonstrated that expanding the language of Peano arithmetic with the function symbol for subtraction already guarantees that the Weak Kleene fixed-point is reached at $\omck$.} 

What is the impact of these non-absoluteness phenomena on axiomatic theories of truth? So far this question has received very little attention in the literature, although some partial answers have been given by \citet{Fuji10}. In this context, let us mention two important results from Fujimoto's paper.

Firstly, Theorem 50 (2) on p. 35 of the quoted paper establishes that $WKF$ and $KF$ have the same arithmetical consequences.\footnote{Namely, it is demonstrated that the proof-theoretic strength of both theories is that of $RA_{< \epsilon_0}$, that is, of ramified analysis up to $\epsilon_0$.} However, the result comes with an important limitation: Fujimoto assumes that the truth axioms of $WKF$ are added to Peano arithmetic \textit{as formulated in the language of primitive recursive arithmetic} (indeed, the availability of terms for primitive recursive functions is employed in the proof of Fujimoto's theorem). As we have seen, the choice of coding and language influences the properties of model-theoretic constructions based on Weak Kleene logic. It would be interesting to know whether (and how) these non-absoluteness phenomena influence the arithmetical strength of Weak Kleene truth axioms. At present we are not aware of any results in this direction. 

Secondly, Theorem 50 (1) on p. 35 establishes that $KF$ can define the truth predicate of $WKF$. This positive result does not depend on the choice of coding and language. Whatever choices are made in this respect, there will be a formula defining in $KF$ the truth predicate satisfying the Weak Kleene axioms.

However, it has been an open question whether $WKF$ can define the truth predicate of $KF$.\footnote{See also \cite[p. 263]{Hal11}, where this question is explicitly stated as open. So far, the only examples suggesting that Weak Kleene based theories of truth may really be weaker than their Strong Kleene counterparts come from the analysis of typed and non-inductive versions of $KF$ and $WKF$, denoted respectively as $PT^-$ and $WPT^-$ (the theories resembling $KF$ and $WKF$, except that their axioms characterise truth for arithmetical sentences only and they do not contain induction for formulas with the truth predicate). Namely, in \citep{MaBa18} it has been demonstrated that the theory $PT^- + INT$, extending $PT^-$ with the axiom of internal induction, is not truth definable in $WPT^- + INT$.} It has been also far from clear whether the unique answer to this question exists. It could happen, after all, that some versions of $WKF$ (say, one built over the language of primitive recursive arithmetic) define the truth predicate of $KF$, while others do not.

Here we are going to demonstrate that the above question has the absolute negative answer. This follows immediately from Theorem \ref{mainthe}, which is the main result of this paper. The upshot is that no matter how coding/language is chosen, $WKF$ does not define the truth predicate of $KF$. Accordingly, throughout the paper we will simply assume that the arithmetical part of $\lt$ (the language of both $KF$ and $WKF$) is fixed without stipulating what sort of terms for primitive recursive functions it contains. The coding will be also treated as fixed but arbitrary. Nevertheless, Cain's and Damnjanovic's result on closure ordinals will play an essential role in the proof to be presented. 

\section{Preliminaries}

This section describes a few miscellaneous observations and results which will be useful to us later on in this paper. 

We start by defining a notion of a diagonal formula. In what follows $\overline{s}$ is a numeral denoting a number $s$. Expressions `$x = name(y)$' and `$x = sub(y, z, s)$' are arithmetical formulas representing in $PA$ the recursive relations `$x$ is a numeral denoting $y$' and `$x$ is the result of substituting (a term) $s$ for (a variable) $z$ in (a formula) $y$'.
\begin{definition}\label{diag}
Let $\varphi(x_1 \ldots x_n, y)$ be an arbitrary formula of $\lt$. We say that $\psi(x_1 \ldots x_n)$ is the diagonal formula for $\varphi(x_1 \ldots x_n, y)$ if it is constructed in the following manner:
\begin{itemize}
\item Denote by $F(x_1 \ldots x_n, y)$ the formula $`\exists ab [a = name(y) \wedge b = sub(y, \ulcorner y \urcorner, a) \wedge \varphi(x_1 \ldots x_n, b)]$',
\item Let $m$ be the G\"odel number of $F(x_1 \ldots x_n, y)$,
\item Define $\psi(x_1 \ldots x_n)$ as $F(x_1 \ldots x_n, \overline{m})$.
\end{itemize}
\end{definition}

It is a well-known fact that the equivalence of $\psi(x_1 \ldots x_n)$ with $\varphi(x_1 \ldots x_n, \overline{\ulcorner \psi(x_1 \ldots x_n)\urcorner})$ is provable already in $PA$ as formulated in $\lt$ (the acronym $PAT$ will be used here for this theory).\footnote{Since by assumption $\varphi(x_1 \ldots x_n, y)$ is an arbitrary formula of $\lt$, the language of $PA$ needs to be extended. However, in order to obtain the equivalence in question it is enough to assume that the predicate `$T(x)$' appears only in the logical axioms of our theory. This is exactly what we mean by `$PA$ as formulated in $\lt$'.}  The next lemma states that the equivalence remains provable (in $KF$) after the truth predicate is applied on both sides. 

\begin{lemma}\label{diagkf}
Let $\psi(x_1 \ldots x_n)$ be the diagonal formula for $\varphi(x_1 \ldots x_n, y)$. Then:
\begin{itemize}
\item[(a)] $KF \vdash T\big(\psi(x_1 \ldots x_n)\big) \equiv T\big(\varphi(x_1 \ldots x_n, \overline{\ulcorner \psi(x_1 \ldots x_n)\urcorner})\big)$,
\item[(b)] $KF \vdash T\big(\neg \psi(x_1 \ldots x_n)\big) \equiv T\big(\neg \varphi(x_1 \ldots x_n, \overline{\ulcorner \psi(x_1 \ldots x_n)\urcorner})\big)$.
\end{itemize}
\end{lemma}
\begin{proof}
The assumption that $\psi(x_1 \ldots x_n)$ is the diagonal formula provides us with an exact information about its form.\footnote{Indeed, this piece of information is absolutely crucial. Observe that for Lemma \ref{diagkf} it is not enough to assume that $\psi(x_1 \ldots x_n)$ is provably equivalent (even in $PAT$) to $\varphi(x_1 \ldots x_n, \overline{\ulcorner \psi(x_1 \ldots x_n)\urcorner})$. Thus, for an arbitrary sentence $\varphi$, $PAT \vdash (\varphi \vee \neg \varphi) \equiv 0 = 0$. Still, for many sentences $\varphi$ (those lacking a determinate truth value) it will not be the case that $KF \vdash T(\varphi \vee \neg \varphi) \equiv T(0 = 0)$. } For (b), fixing $x_1 \ldots x_n$ we observe that the following conditions are provably equivalent in $KF$:

\begin{enumerate}
\item $T\big(\neg \psi(x_1 \ldots x_n)\big)$,
\item $T\big(\neg F(x_1 \ldots x_n, \overline{m})\big)$,
\item $\forall ab [a = name(\overline{m}) \wedge b = sub(\overline{m}, \ulcorner y \urcorner, a) \rightarrow T\big(\neg \varphi(x_1 \ldots x_n, b)\big)]$,
\item $T\big(\neg \varphi(x_1 \ldots x_n, \overline{\ulcorner F(x_1 \ldots x_n, \overline{m})\urcorner})\big)$,
\item $T\big(\neg \varphi(x_1 \ldots x_n, \overline{\ulcorner \psi(x_1 \ldots x_n)\urcorner})\big)$.
\end{enumerate}
The equivalence of 1 and 2 holds by the definition of $\psi$; the second equivalence holds by the axiom of $KF$ for negated existential statements and the fact that provably in $KF$, we have full disquotation for arithmetical formulas. The third equivalence holds because, provably in $PA$, $\forall ab [a = name(\overline{m}) \wedge b = sub(\overline{m}, \ulcorner y \urcorner, a) \rightarrow b = \overline{\ulcorner F(x_1 \ldots x_n, \overline{m}) \urcorner}]$. The last equivalence holds by the definition of $\psi$.

Part (a) is proved in a very similar manner.
\end{proof}

The following observation will be also useful.

\begin{observation}\label{positive} Let $\tau(x)$ be a $KF$ truth predicate in $(N, \twk)$. For $\psi \in \sentt$, denote by $\psi^{\tau}$ the result of replacing all occurrences of `$T(t)$' by `$\tau(t)$'  in $\psi$. Then for every positive sentence $\psi \in \lt$:\footnote{A positive sentence is defined as a sentence in which every occurrence of `$T$' lies in the scope of even number of negations.}
\begin{center}
$(N, \twk) \models \tau(\psi) \equiv \psi^{\tau}$.
\end{center}
\end{observation}

The proof is a minor modification of the familiar reasoning showing that $KF$ proves disquotation for positive formulas. For more details, see \cite[p. 201]{Hal11}.

The next definition introduces the notion of a grounded sentence. 

\begin{definition}\label{grounded} Let $E$ be either Weak Kleene or Strong Kleene evaluation scheme. We say that a sentence $\psi \in \lt$ is E-grounded iff either $\psi$ or $\neg \psi$ belongs to $T^E$.
\end{definition}

We state now the following fact.

\begin{fact}\label{basic} \hfil
\begin{itemize}
\item[(i)] For every $\psi \in \lt$, if $\psi \in \twk$, then $\psi \in \tsk$,\footnote{Both here and elsewhere in such contexts we tacitly assume the sameness of coding in the Weak Kleene and the Strong Kleene model.}
\item[(ii)] For every $\psi \in \lt$, if $\psi \in \tsk$, then for every $T \subseteq N~($if $(N, T) \models KF$, then $\psi \in T)$,
\item[(iii)] For every $T \subseteq N~\big($if $(N, T) \models KF$, then $\neg \exists \psi (\psi$ is SK-grounded and $\psi \in T$ and $\neg \psi \in T)\big)$.
\end{itemize}
\end{fact}
Part (i) can be easily proved by ordinal induction. Parts (ii) and (iii) are well-known.

Fact \ref{basic} permits us to derive the following useful corollary.

\begin{corollary}\label{tauwk}
If $\tau(x)$ is a $KF$ truth predicate in $(N, \twk)$, then for every WK-grounded sentence $\psi \in \lt$:
\begin{center}
$(N, \twk) \models \tau(\psi) \equiv T(\psi)$.
\end{center}
\end{corollary}
\begin{proof}
Define $T$ as $\{x: (N, \twk) \models \tau(x) \}$. Given that $\tau(x)$ is a $KF$ truth predicate in $(N, \twk)$, it is easy to observe that $(N, T) \models KF$. Fixing an arbitrary WK-grounded $\psi$, our task is to show that $(N, \twk) \models \tau(\psi) \equiv T(\psi)$. Note that since $\psi$ is WK-grounded, it is also SK-grounded by Fact \ref{basic}(i). 

For the implication from left to right, assume that $(N, \twk) \models \tau(\psi)$, so $(N, T) \models T(\psi)$ by the definition of $T$. Assuming for the indirect proof that $(N, \twk) \models \neg T(\psi)$, we obtain: $(N, \twk) \models T(\neg \psi)$ because $\psi$ is grounded. Therefore by Fact \ref{basic}(i) $(N, \tsk) \models T(\neg \psi)$ and so $(N, T) \models T(\neg \psi)$  by Fact \ref{basic}(ii). In effect, both $\psi$ and $\neg \psi$ belong to $T$. However, this contradicts Fact \ref{basic}(iii), because $\psi$ is SK-grounded.

For the opposite implication, assume that $(N, \twk) \models T(\psi)$. Then by Fact \ref{basic}(i) we have: $(N, \tsk) \models T(\psi)$ and so $(N, T) \models T(\psi)$ by Fact \ref{basic}(ii). From the definition of $T$, it immediately follows that $(N, \twk) \models \tau(\psi)$ as required.
\end{proof}

\section{Statement of the result. Main lemmas}

In this section we state our principal theorem, formulate the main lemmas and provide the proof of one of the two lemmas. 

\begin{theorem}[main theorem]\label{mainthe} There is no formula $\tau(x) \in \lt$ such that $\tau(x)$ is a truth predicate of $KF$ in $(N, \twk)$.
\end{theorem}

\noindent
Since $(N, \twk) \models WKF$, it immediately follows that $WKF$ does not define the truth predicate of $KF$.

Theorem \ref{mainthe} will be obtained as a direct corollary from the two lemmas listed below.

\begin{lemma}\label{main1}
There are no formulas $\tau(x)$ and $G(x) \in \lt$ such that $\tau(x)$ is a $KF$ truth predicate in $(N, \twk)$ and for every $\psi \in \sentt$:
\begin{itemize}
\item[(a)] $\psi$ is WK-grounded iff $(N, \twk) \models \tau(G(\psi))$,
\item[(b)] $\psi$ is WK-ungrounded iff $(N, \twk) \models \tau(\neg G(\psi))$.
\end{itemize}
\end{lemma}

\begin{lemma}\label{main2} If $(N, \twk)$ defines a truth predicate of $KF$, then there are formulas $\tau(x)$ and $G(x)$ satisfying the conditions (a) and (b) from Lemma \ref{main1}.
\end{lemma}

Observe that Theorem \ref{mainthe} follows trivially from the lemmas. We now give the proof of Lemma \ref{main1}. Lemma \ref{main2} will be proved in the sections to follow. \\

\noindent
\textbf{Proof of Lemma \ref{main1}.}
Assuming that the lemma is false, define:
\begin{center}
$Tr(x):= \tau(x[G(t) \wedge T(t)/T(t)])$.
\end{center}
In other words, $Tr(x)$ is the formula stating that the result of formally substituting `$G(t) \wedge T(t)$' for every occurrence of `$T(t)$' in $x$ satisfies $\tau$.\footnote{I use the word `formally' in order to emphasise that the substitution is carried out in the object language, not in the metalanguage.} 

We now show that for every $\psi \in \sentt$, $(N, \twk) \models Tr(\psi) \equiv \psi$, which contradicts Tarski's undefinability theorem, thus ending the proof. The argument proceeds by induction on the complexity of $\psi$. Below we consider only cases of $\psi$ of the form $T(t)$ or $\neg T(t)$, since it is only here where conditions (a) and (b) from Lemma \ref{main1} are crucially used.

For $\psi = \ulcorner T(t) \urcorner$, our task is to demonstrate that:
\begin{center}
$(N, \twk) \models \tau(G(t) \wedge T(t)) \equiv T(t)$. 
\end{center}
If $val(t)$ is WK-ungrounded, then the above equivalence holds in $(N, \twk)$ because both sides of the equivalence are false. In other words, we have then: $(N, \twk) \nvDash T(t)$ and $(N, \twk) \nvDash \tau(G(t) \wedge T(t))$. The fact that $(N, \twk) \nvDash T(t)$ is obvious on the assumption that $val(t)$ is WK-ungrounded, so for the indirect proof let us assume that $(N, \twk) \models \tau(G(t) \wedge T(t))$. Then  $(N, \twk) \models \tau(G(t))$, which by condition (a) of the lemma implies that $val(t)$ is WK-grounded and in this way a contradiction is obtained.

On the other hand, if $val(t)$ is WK-grounded, then the following conditions are equivalent:
\begin{itemize}
\item $(N, \twk) \models \tau(G(t) \wedge T(t))$,
\item $(N, \twk) \models \tau(T(t))$,
\item $(N, \twk) \models T(t)$.
\end{itemize}
The first equivalence holds by the assumption (a) of the lemma. For the second equivalence, observe that if $val(t)$ is WK-grounded, then $T(t)$ is also WK-grounded and so by Corollary \ref{tauwk} we have: $(N, \twk) \models \tau(T(t))$ iff $(N, \twk) \models T(T(t))$, with the last condition being equivalent to $(N, \twk) \models T(t)$ by the axiom {\scshape wkf12}.

For $\psi = \ulcorner \neg T(t) \urcorner$, our task is to demonstrate that:
\begin{center}
$(N, \twk) \models \tau \big(\neg (G(t) \wedge T(t)) \big) \equiv \neg T(t)$. 
\end{center}
If $val(t)$ is WK-ungrounded, then the above equivalence holds in $(N, \twk)$ because both sides of the equivalence are true. In other words, we have then: $(N, \twk) \models \neg T(t)$ and $(N, \twk) \models \tau \big(\neg (G(t) \wedge T(t)) \big)$. For the second conjunct (the first one is obvious) observe that by the assumption (b) of the lemma we will have then $(N, \twk) \models \tau(\neg G(t))$ and thus $(N, \twk) \models \tau \big(\neg (G(t) \wedge T(t)) \big)$ because $\tau(x)$ is a $KF$ truth predicate.

If $val(t)$ is WK-grounded, then the following conditions are equivalent:
\begin{itemize}
\item $(N, \twk) \models \tau \big(\neg (G(t) \wedge T(t)) \big)$,
\item $(N, \twk) \models \tau(\neg T(t))$,
\item $(N, \twk) \models \neg T(t)$.
\end{itemize}
The first equivalence holds because by the assumption (b) of the lemma $(N, \twk) \nvDash \tau (\neg G(t))$.\footnote{The condition `$(N, \twk) \models \tau \big(\neg (G(t) \wedge T(t)) \big)$' implies that $(N, \twk) \models \tau \big(\neg G(t)\big)$ or $(N, \twk) \models \tau \big(\neg T(t)\big)$. Since $val(t)$ is WK-grounded, the first disjunct is excluded by the assumption (b) of the lemma, hence $(N, \twk) \models \tau(\neg T(t))$.} With $val(t)$ being WK-grounded, `$\neg T(t)$' is also WK-grounded, hence the second equivalence follows immediately from Corollary \ref{tauwk}. $\Box$

\section{Proof of Lemma \ref{main2}}

For the proof of Lemma \ref{main2} some additional auxiliary notions and facts will be needed. These are introduced below.

\begin{definition}\label{tria} \hfil
\begin{itemize}
\item $x \vartriangleleft y$ is an abbreviation of the following arithmetical formula:
\begin{flushleft}
$\sentt(x) \wedge \sentt (y) \wedge \newline
\Big( \exists t \in Tm^c  (y = \ulcorner T(t) \urcorner \wedge x = val(t)) \newline
\vee \exists \psi \in \sentt (y = \ulcorner \neg \psi \urcorner \wedge x = \psi) \newline
\vee \exists \varphi, \psi \in \sentt(y = \ulcorner \varphi \circ \psi \urcorner \wedge x = \varphi \vee x = \varphi) \newline
\vee  \exists \theta(x) \in Fm_T^{\leq 1} \exists t \in Tm^c \exists v \in Var(y = \ulcorner Qv \theta(v) \urcorner \wedge x = \ulcorner \theta(t) \urcorner) \Big)$.
\end{flushleft}
\item $\vartriangleleft^*$ denotes the transitive closure of $\vartriangleleft$ (in other words, $x \vartriangleleft^* y$ iff there is a path from $y$ to $x$ along the $\vartriangleleft$ relation),
\item $x \trianglelefteq y$ iff $x \vartriangleleft y \vee x = y$; $x \trianglelefteq^* y$ is defined in a similar manner.
\item a sentence $\psi \in \lt$ is well-founded (wf in short) iff $\tria$ is well-founded on $\{x: x \etrias \psi \}$.
\end{itemize}
\end{definition}
The expression `$Fm_T^{\leq 1}$' used in the above definitions denotes the set of formulas of $\lt$ with at most one free variable. $Var$ is the set of variables.
\begin{definition}\label{ord} For an arbitrary well-founded $\psi$, $Ord(\psi)$ (the ordinal of $\psi$) is defined in the following manner:
\begin{center}
$Ord(\psi) = sup\{Ord(x): x \tria \psi \}$.
\end{center}
\end{definition}

\begin{lemma}\label{grwf}
For every $\psi \in \sentt$, $\psi$ is $WK$-grounded iff $\psi$ is well-founded.
\end{lemma}
\begin{proof} Both implications are proved by ordinal induction. For the implication from left to right, fix $\alpha$ and assume that $\forall \beta < \alpha \forall \psi \in \sentt [\psi \in T^{WK}_{\beta} \rightarrow \psi$ is well-founded]. The claim then is that $\forall \psi \in \sentt [\psi \in T^{WK}_{\alpha} \rightarrow \psi$ is well-founded]. For $\alpha$ being a limit ordinal, the claim follows trivially from the inductive assumption. For $\alpha$ of the form $\beta + 1$, the claim is also easily obtained by analysing cases from Definition \ref{wkjump}. For example, if $\psi = \ulcorner Qv \varphi(v) \urcorner$, then for every $t \in Tm^c$, either $\varphi(t)$ or $\neg \varphi(t)$ belongs to $T^{WK}_{\beta}$ and so by the inductive assumption for every $t \in Tm^c$, $\varphi(t)$ is well-founded, from which it follows that $Qv \varphi(v)$ is well-founded.

In the proof of the opposite implication, we show that 

\begin{center}
$\forall \alpha \forall \psi \in \sentt [(\psi$ is wf $\wedge~Ord(\psi) = \alpha) \rightarrow \psi~is$ \textit{WK-grounded}].
\end{center}
Fixing $\alpha$ and assuming that this holds for every $\psi$ such that $Ord(\psi) < \alpha$, take an arbitrary sentence $\psi$ such that $Ord(\psi) = \alpha$. If $\alpha = 0$, then $\psi$ is of the form `$s_1 = s_2$' or it has the form `$T(t)$' where $val(t)$ is not a sentence. Then either $\psi$ or $\neg \psi$ belongs to $T^{WK}_1$ and thus $\psi$ is WK-grounded. If $\alpha$ is a limit ordinal, then $\psi$ has the form $Qv \varphi(v)$ and since by the inductive assumption all sentences of the form $\varphi(t)$ are WK-grounded, it is easy to observe that $\psi$ is also WK-grounded. The case of $\alpha = \beta +1$ is also straightforward and we leave it to the reader.
\end{proof}

The next lemma characterises the bounds for ordinals of sentences determined as true in the Weak Kleene hierarchy.

\begin{lemma}\label{bounds} \hfil
\begin{itemize}
\item[(a)] $\forall \varphi \in \sentt \forall \alpha \big( \varphi \in T^{WK}_{\alpha +1} \setminus T^{WK}_{\alpha} \rightarrow Ord(\varphi) \geq \alpha \big)$.
\item[(b)] $\forall \varphi \in \sentt \forall \alpha > 0 \big( \varphi \in T^{WK}_{\alpha} \rightarrow Ord(\varphi) < \omega \cdot \alpha \big)$.
\end{itemize}
\end{lemma}

Both (a) and (b) are proved by ordinal induction, with the bulk of the proofs devoted to analysing cases. \\

\noindent
\textbf{Proof of (a)}. Assume that $\forall \beta < \alpha \forall \varphi \in \sentt  \big( \varphi \in T^{WK}_{\beta +1} \setminus T^{WK}_{\beta} \rightarrow Ord(\varphi) \geq \beta \big)$.  Fix $\varphi \in T^{WK}_{\alpha +1}$ such that $\varphi \notin T^{WK}_{\alpha}$. Our claim is that $Ord(\varphi) \geq \alpha$. From now on, the proof proceeds by considering all possible forms of $\varphi$. Below we discuss just two cases, leaving the rest of them to the reader.

Case 1: $\varphi$ is of the form `$T(t)$'. Since $\varphi \in T^{WK}_{\alpha +1}$, we have: $val(t) \in T^{WK}_{\alpha}$. Since $\varphi \notin T^{WK}_{\alpha}$, we have also: $\forall \beta < \alpha~val(t) \notin T^{WK}_{\beta}$. Therefore $\alpha = 0$ or $\alpha$ is a successor number of the form $\beta + 1$ (no new sentences can be added on limit levels). If $\alpha = 0$, then obviously $Ord(\varphi) \geq \alpha$. If $\alpha = \beta + 1$, then we have: $val(t) \in T^{WK}_{\beta + 1}$ and $val(t) \notin T^{WK}_{\beta}$, so by our inductive assumption $Ord(val(t)) \geq \beta$, hence $Ord(T(t)) \geq \beta+1$ as required.

Case 2: $\varphi$ is of the form `$\forall x \psi(x)$'. Since $\varphi \in T^{WK}_{\alpha +1}$, we have: $\forall t \in Tm^c~\psi(t) \in T^{WK}_{\alpha}$. We now consider two possibilities: either  (i) there is a constant term $t$ such that $\psi(t) \in  T^{WK}_{\alpha}$ and $\psi(t)$ does not appear anywhere earlier in the hierarchy, or (ii) no such term exists. In the first case, $\alpha$ is a successor ordinal of the form $\beta+1$, so by the inductive assumption $Ord(\psi(t)) \geq \beta$, therefore $Ord(\varphi) \geq \alpha$ and the proof is done. In the second case $\alpha$ must be a limit ordinal.\footnote{If $\alpha$ were equal to $\beta + 1$, then by (ii) we would have: $\forall t \in Tm^c \psi(t) \in T^{WK}_{\beta}$ and thus $\varphi \in T^{WK}_{\alpha}$, contrary to the main assumption of the proof.} Then we obtain:
\begin{center}
$\forall \beta < \alpha \exists \gamma > \beta \big( \gamma < \alpha \wedge \exists t (\psi(t) \in T^{WK}_{\gamma +1} \wedge \psi(t) \notin T^{WK}_{\gamma}) \big).$
\end{center} 
Otherwise, fixing $\beta < \alpha$ and assuming that $\gamma$ with the above properties does not exist, we would have: $\forall t \in Tm^c \psi(t) \in T^{WK}_{\beta}$ and thus $\varphi \in T^{WK}_{\alpha}$, contrary to the assumption of the proof.

By the inductive assumption, it now follows that:
\begin{center}
$\forall \beta < \alpha \exists \gamma > \beta \big( \gamma < \alpha \wedge \exists t Ord(\psi(t)) \geq \gamma \big)$.
\end{center}
In effect, $sup \{Ord(\psi(t)): t \in Tm^c \} \geq \alpha$ and thus $Ord(\varphi) \geq \alpha$. $\Box$ \\

\noindent
\textbf{Proof of (b)}. Our inductive assumption is that $\forall \beta < \alpha \big( \beta > 0 \rightarrow \forall \varphi \in \sentt ( \varphi \in T^{WK}_{\beta} \rightarrow Ord(\varphi) < \omega \cdot \beta ) \big)$. Now, assuming that $\alpha > 0$ and fixing $\varphi \in T^{WK}_{\alpha}$, we claim that $Ord(\varphi) < \omega \cdot \alpha$.

It is easy to observe that for every arithmetical sentence $\psi$ (with no occurrence of the truth predicate), $Ord(\varphi) < \omega$. Moreover, if our fixed $\varphi$ belongs to $T^{WK}_{1}$, then $\varphi$ is either arithmetical or it is of the form `$T(t)$' or `$\neg T(t)$'. In all of these cases $Ord(\varphi) < \omega$, hence if $\alpha = 1$, then $Ord(\varphi) < \omega \cdot \alpha$.

If $\alpha$ is a limit ordinal, the result follows trivially from the inductive assumption. So assume that $\alpha = \beta + 1$ with $\beta > 0$. From now on, the proof proceeds by analysing all possible forms of $\varphi$. We will consider just two cases, leaving the rest of them to the reader.

Case 1: $\varphi$ is of the form $\neg T(t)$. Then $val(t)$ is not a sentence or $\neg val(t) \in T^{WK}_{\beta}$. In the first case $Ord(\varphi) = 1$ and we are done. In the second case by the inductive assumption, $Ord(\neg val(t)) < \omega \cdot \beta$, hence $Ord(\neg T(t)) < \omega \cdot (\beta + 1)$. 

Case 2: $\varphi$ is of the form $\forall x \psi(x)$. Then for all $t \in Tm^c$, $\psi(t) \in T^{WK}_{\beta}$ and so by the inductive assumption $\forall t \in Tm^c~Ord(\psi(t)) < \omega \cdot \beta$. Since by definition $Ord(\varphi) = sup \{Ord(\psi(t)): t \in Tm^c \}$, we have: $Ord(\varphi) \leq \omega \cdot \beta$. Therefore $Ord(\varphi) < \omega \cdot (\beta +1)$, which means that $Ord(\varphi) < \omega \cdot \alpha$. $\Box$

\begin{observation}\label{compord} \hfil
\begin{itemize}
\item[(a)] $\forall s, s' [s \trias s' \rightarrow \exists s_1 (s_1 \tria s' \wedge s \etrias s_1)]$,
\item[(b)] $\forall s \forall \alpha [Ord(s) = \alpha \rightarrow \forall \beta < \alpha \exists s' \trias s~Ord(s') = \beta]$.
\end{itemize}
\end{observation}
\begin{proof}
Part (a) follows directly from the definition of $\trias$. For an indirect proof of (b), assume that $Ord(s) = \alpha$ and let $\beta < \alpha$ be such that $\neg \exists s' \trias s~Ord(s') = \beta$. Choose the least ordinal $\gamma$ such that $\gamma > \beta$ and $\exists s' \etrias s~Ord(s') = \gamma$ (such an ordinal exists because $\alpha$ satisfies both conditions). Fixing $s'$ such that $Ord(s') = \gamma$, we obtain:
\begin{center}
$\gamma = sup\{Ord(s'': s'' \tria s' \}$.
\end{center}
Since $\gamma > \beta$, we have: $\gamma \ne 0$. We observe that $\gamma$ cannot be a successor number, because if $\gamma = \delta +1$, then $\delta \in \{Ord(s''): s'' \tria s' \}$, which contradicts the choice of $\gamma$ as the least ordinal with the stipulated property. It follows that $\gamma$ is a limit ordinal. Therefore there is an ordinal $\delta > \gamma$ and $s''$ such that $Ord(s'') = \delta$ and $s'' \tria s'$, which again contradicts the choice of $\gamma$.
\end{proof}

The next two sections contain a proof by cases of Lemma \ref{main2}. Assuming that $(N, \twk)$ defines a truth predicate of $KF$, we take into account the possible closure ordinals of $(N, \twk)$ (see Definition \ref{clord}). By Theorem \ref{theclord} there are two possibilities: either this ordinal is smaller than $\omck$ or it is $\omck$. It transpires that in both cases Lemma \ref{main2} can be proved, although the means used in the proofs are quite different.

\subsection{Case 1: $ClOrd_{(N, \twk)} < \omega_1^{CK}$}

The key observation is that in this case ordinals of well-founded sentences can be restricted by a fixed number below $\omck$.

\begin{observation}\label{ordpsi} $\exists \kappa < \omck~\forall \psi \in \sentt \big(\psi$ is well-founded $\rightarrow Ord(\psi) < \kappa \big)$.
\end{observation}
\begin{proof}
Let $\alpha$ be the closure ordinal of $(N, \twk)$. Define $\kappa$ as $\omega \cdot \alpha$. Since $\alpha < \omega_1^{CK}$, so is $\kappa$. Let $\psi$ be well-founded, so by Lemma \ref{grwf} either $\psi$ or its negation belongs to $T^{WK}_{\beta}$ for some $\beta < \alpha$. Then by Lemma \ref{bounds}(b), $Ord(\psi) < \omega \cdot \beta$ and thus $Ord(\psi) < \kappa$.
\end{proof}

The proof of Lemma \ref{main2} employs the following theorem.\footnote{For the original formulation of the theorem, see \cite[p. 715-716]{Kri75}. For the proof, see \cite[p. 108]{Schi15}, Theorem 7.2.8.}

\begin{theorem}[Kripke 1975]\label{kripthe} Let $S$ be an arbitrary set of natural numbers. The following conditions are equivalent:
\begin{itemize}
\item $S$ is $\Delta_1^1$,
\item there is a formula $\varphi(x) \in \lt$ such that $\varphi(x)$ is total in $(N, \tsk)$ and $S = \{n: (N, \tsk) \models T(\varphi(n)) \}$.
\end{itemize}
\end{theorem}

In view of Theorem \ref{kripthe}, it is enough to show that the set of WK-grounded sentences (that is, the set of well-founded sentences) is $\Delta_1^1$. 

\begin{lemma}\label{delta11} The set of sentences well-founded in $(N, \twk)$ is $\Delta_1^1$.
\end{lemma}
\begin{proof}
The $\Pi_1^1$ formulation is straightforward. For starters, define `$X$ is a path on $y$' as the conjunction of the following four arithmetical formulas with one second-order free variable:
\begin{itemize}
\item $\forall x \in X \exists a, b (b \etrias y \wedge x = (a,b))$
\item $(0, y) \in X$
\item $\forall a,b,a',b' ((a,b) \in X \wedge (a',b') \in X \rightarrow b \etrias b' \vee b' \etrias b)$ 
\item $\forall a, b ((a+1, b) \in X \rightarrow \exists c \tria b (a,c) \in X) $.
\end{itemize}
Then we can express the well-foundedness of an arbitrary $\psi$ by means of the following $\Pi_1^1$ formula:
\begin{center}
$\forall X [X$ is a path on $\psi \rightarrow \exists n \forall x \in X x < n]$.
\end{center}

For a $\Sigma_1^1$ formulation, let $\kappa$ be an ordinal smaller that $\omck$ whose existence is guaranteed by Observation \ref{ordpsi}. Then the well-foundedness of an arbitrary $\psi$ can be expressed by means of the following $\Sigma_1^1$ formula:
\begin{center}
$\exists f \exists \alpha < \kappa \big(f$ is a surjection mapping $\{x: x \etrias \psi \}$ onto $\alpha \wedge \forall x, y \etrias \psi (x \tria y \rightarrow f(x) < f(y)) \big)$.
\end{center}
Note that since $\kappa < \omck$, we can treat the quantification over ordinals as quantification over ordinal notations.
\end{proof}

Now we immediately obtain the proof of Lemma \ref{main2}. \\

\noindent
\textbf{Proof of Lemma \ref{main2}} Let $\tau(x)$ be a $KF$ truth predicate in $(N, \twk)$. By Lemma \ref{delta11} and Theorem \ref{kripthe}, let $G(x)$ be a formula total in $(N, \tsk)$ such that the set of WK-grounded sentences can be characterised as  $\{\psi: (N, \tsk) \models T(G(\psi)) \}$. The totality condition means that for every $t$, the formula `$G(t)$' is SK-grounded and so the conditions (a) and (b) of Lemma \ref{main2} follow easily from Fact \ref{basic}(ii) and (iii). $\Box$

\subsection{Case 2: $ClOrd_{(N, \twk)} = \omega_1^{CK}$}

In this case, the reasoning from the previous subsection is clearly inapplicable, since by Lemma \ref{bounds}(a) we cannot restrict the ordinals of well-founded formulas by any ordinal below $\omck$. In effect, our strategy here will be quite different. Indeed, now the proof of Lemma \ref{main2} is based on the insight that well-founded sentences form a structure which is complex enough to permit us (in the presence of the $KF$ truth predicate) to reconstruct in $(N, \twk)$ various Strong Kleene model-theoretic constructions, including the Kripkean construction of the least fixed-point model.

Let us start by the following definition. Intuitively, the formula `$\psi(s, x)$' defined below is designed to express that $x$ is a sentence determined as true at the ordinal level $s$ of the Strong Kleene least fixed-point construction. However, ultimately instead of ordinals we will be using the well-founded sentences of $\lt$.

\begin{definition}\label{psi} Let $\psi$ be the diagonal formula satisfying (provably in PAT) the condition:
\begin{align}
	\psi(s, x) \equiv &\ \ s \in \sentt~\wedge  \nonumber\\
	 &\ \  \Big( x \in \sentpa \wedge T(x)\nonumber\\
	\vee&\ \  x = \ulcorner T(t) \urcorner \wedge \exists s' \vartriangleleft s T(\psi(s', val(t))) \nonumber\\
	\vee&\ \  x = \ulcorner \neg T(t) \urcorner \wedge \big( \exists s' \vartriangleleft s T(\psi (s', \neg val(t))) \vee \neg \sentt(val(t)) \big) \nonumber\\
	\vee&\ \  x = \ulcorner \neg \neg \varphi \urcorner \wedge  \exists s' \vartriangleleft s T ( \psi (s', \varphi))  \nonumber\\
	\vee&\ \  x = \ulcorner \varphi \circ \chi \urcorner \wedge  \exists s's'' \vartriangleleft s \big( T ( \psi (s', \varphi)) \circ T ( \psi (s'', \chi)) \big) \nonumber\\
		\vee&\ \  x = \ulcorner \neg (\varphi \circ \chi ) \urcorner \wedge  \exists s's'' \vartriangleleft s \big( T ( \psi (s', \neg \varphi)) \circ_d T ( \psi (s'', \neg \chi)) \big) \nonumber\\
				\vee&\ \  x = \ulcorner Qv \varphi  \urcorner \wedge  \exists s' \vartriangleleft s Qa T ( \psi (s', \varphi(a))) \nonumber\\
\vee&\ \  x = \ulcorner \neg Qv \varphi  \urcorner \wedge  \exists s' \vartriangleleft s Q_d a T ( \psi (s', \neg \varphi(a))) \Big) \nonumber\
\end{align}

\end{definition}

In what follows we assume that $\tau(x)$ is a $KF$ truth predicate in $(N, \twk)$.  Denoting by `$F(s, x)$' the formula on the right side of the biconditional in Definition \ref{psi}, we obtain the following corollary:

\begin{corollary}\label{psiwb} \hfil
\begin{itemize}
\item $(N, \twk) \models \tau(\psi(s, x)) \equiv \tau(F(s, x))$,
\item $(N, \twk) \models \tau(F(s, x)) \equiv F^{\tau}(s, x)$.
\end{itemize}
\end{corollary} 

The first part follows by Lemma \ref{diagkf}, the second by Observation \ref{positive} and the fact that $\psi(s, x)$ is positive. In effect, when working in $(N, \twk)$, we can always move freely between `$\tau(\psi(s, x))$' and the result of substituting `$\tau(t)$' for all the occurrences of `$T(t)$' in  `$F(s, x)$'.

We write `$s \in D$' ($s$ is determined) as an abbreviation of `$T(s) \vee T(\neg s)$'. We denote by `$D^{WK}$' the set of WK-grounded sentences. Now we formulate the following basic observation.

\begin{observation}\label{ari}
For every $\varphi \in \sentpa$, for every $s \in D^{WK} \big( (N, \twk) \models \tau(\psi(s, \varphi))$ iff $\varphi \in Th(N) \big)$.
\end{observation}
\begin{proof} For the implication from right to left, observe that since $\tau(x)$ is a $KF$ truth predicate, for any $\varphi \in Th(N)$ we will have $(N, \twk) \models \varphi \in \sentpa \wedge \tau(\varphi)$. Then by Corollary \ref{psiwb} together with the definition of $\psi(s,x)$ it follows that $(N, \twk) \models \tau(\psi(s, \varphi))$.

For the opposite implication, let $\alpha$ be the least ordinal such that for some $s$ we have: $Ord(s) = \alpha$, $(N, \twk) \models \tau(\psi(s, \varphi))$ but $\varphi \notin Th(N)$. By considering all possible forms of the formula $\varphi$ we note that there has to be an $s' \tria s$ such that for some sentence $\chi \notin Th(N)$ $(N, \twk) \models \tau(\psi(s', \chi))$.\footnote{It is easy to observe that $\varphi$ cannot have the form `$t = s$', therefore $\varphi$ has to be of the form $\neg \neg \chi$, $\xi \circ \chi$, $\neg (\xi \circ \chi)$, $Qv \chi$ or $\neg Qv \chi$.} But then $Ord(s') < \alpha$, which contradicts the choice of $\alpha$. 
\end{proof}\\
The next lemma establishes the monotonicity of $\psi$ under $\tau$.
\begin{lemma}\label{monoton}
$\forall s s' \in D^{WK} \Big( (Ord(s) \leq Ord(s')) \rightarrow \forall \varphi \in \sentt \big( (N, \twk) \models \tau(\psi(s, \varphi)) \rightarrow (N, \twk) \models \tau(\psi(s', \varphi)) \big) \Big)$.
\end{lemma}
\begin{proof}
Fix an ordinal $\alpha$ and assume that the lemma is true below $\alpha$, that is:
\begin{center}
$\forall \gamma < \alpha \forall s s' \in D^{WK} \Big( (Ord(s) \leq \gamma \wedge Ord(s') = \gamma) \rightarrow$ \\ $\forall \varphi \in \sentt \big( (N, \twk) \models \tau(\psi(s, \varphi)) \rightarrow (N, \twk) \models \tau(\psi(s', \varphi)) \big) \Big)$.
\end{center}
We claim that $\forall s s' \in D^{WK} \Big( (Ord(s) \leq \alpha\ \wedge Ord(s') = \alpha) \rightarrow \forall \varphi \in \sentt \big( (N, \twk) \models \tau(\psi(s, \varphi)) \rightarrow (N, \twk) \models \tau(\psi(s', \varphi)) \big) \Big)$.

Fix $s$ and $s' \in D^{WK}$ such that $Ord(s) \leq \alpha$ and $Ord(s') = \alpha$. Assuming that $(N, \twk) \models \tau(\psi(s, \varphi))$, we are going to show that $(N, \twk) \models \tau(\psi(s', \varphi))$. The proof proceeds by considering cases. For illustration we present below three of them; the reasoning in the remaining cases is very similar.

Case 1: $\varphi \in \sentpa$. Then the conclusion follows by Observation \ref{ari}.

Case 2: $\varphi = (\neg) T(t)$. Then $(N, \twk) \models \exists s_1 \tria s~\tau(\psi(s_1, (\neg) val(t)))$. Fixing such an $s_1$, we observe that $Ord(s_1) < \alpha$. By Observation \ref{compord}(b), choose $s_2 \trias s'$ such that $Ord(s_2) = Ord(s_1)$. Then by the inductive assumption we obtain $(N, \twk) \models \tau(\psi(s_2, (\neg) val(t)))$. By Observation \ref{compord}(a), choose $s_3 \tria s'$ such that $s_2 \etrias s_3$. Then again by the inductive assumption we obtain $(N, \twk) \models \tau(\psi(s_3, (\neg) val(t)))$. Since $s_3 \tria s'$, we finally obtain $(N, \twk) \models \tau(\psi(s', (\neg) T(t)))$ by the definition of $\psi$.

Case 3:  $\varphi = Qv \chi$. Then $(N, \twk) \models \exists s_1 \trias s~Qa~\tau(\psi(s_1, \chi(a)))$. Fix such an $s_1$. Applying Observation \ref{compord}, choose $s_2$ and $s_3$ such that $Ord(s_2) = Ord(s_1)$, $s_2 \trias s'$, $s_2 \etrias s_3$, $s_3 \tria s'$. By the inductive assumption we obtain:
\begin{itemize}
\item[] $(N, \twk) \models Qa~\tau(\psi(s_2, \chi(a)))$,
\item[] $(N, \twk) \models Qa~\tau(\psi(s_3, \chi(a)))$.
\end{itemize}
Since $s_3 \tria s'$, we obtain $(N, \twk) \models \tau(\psi(s', Qv \chi))$ by the definition of $\psi$.
\end{proof}

The lemma below establishes a connection between the behaviour of $\psi$ under $\tau$ and Kripke's hierarchy from Definition \ref{lfp}.
\begin{lemma}\label{leastfp} 
$\forall s \forall \alpha [Ord(s) = \alpha \rightarrow \forall \varphi \in \sentt \big( \varphi \in T_{\alpha}^{SK} \equiv (N, \twk) \models \tau(\psi(s, \varphi)) \big)]$.
\end{lemma}
\begin{proof}
Assume that $\forall \beta < \alpha \forall s [Ord(s) = \beta \rightarrow \forall \varphi \in \sentt \big( \varphi \in T_{\beta}^{SK} \equiv (N, \twk) \models \tau(\psi(s, \varphi)) \big)]$. We claim that the same holds also for $\alpha$. 

If $\alpha = Ord(s) = 0$, then $\varphi \in T_{\alpha}^{SK}$ iff $(N, \twk) \models \tau(\psi(s, \varphi))$ iff $\varphi \in Th(N)$. If $\alpha = \beta + 1$, the proof proceeds by analysing possible forms of $\varphi$. Thus, e.g.:

Case 1. $\varphi = T(t)$. Assuming that $\varphi \in T_{\beta +1}^{SK}$, we obtain: $val(t) \in T_{\beta}^{SK}$. Let $s' \tria s$ be such that $Ord(s') = \beta$. Then by the inductive assumption we have $(N, \twk) \models \tau(\psi(s', \varphi))$ and therefore $(N, \twk) \models \tau(\psi(s, \varphi))$. The proof of the opposite implication is very similar.

Case 2. $\varphi = Qv \chi(v)$. Assuming that $\varphi \in T_{\beta +1}^{SK}$, we obtain: $Qa (\chi(a) \in T_{\beta}^{SK})$. As before, choosing $s' \tria s$ such that $Ord(s') = \beta$ we obtain:  $(N, \twk) \models Qa \tau(\psi(s', \chi(a)))$ and therefore $(N, \twk) \models \tau(\psi(s', Qv\chi(v)))$. Again, the proof of the opposite implication is very similar.
We leave the other cases to the reader.

If $\alpha$ is a limit ordinal, we argue as follows. Assuming that $\varphi \in T_{\alpha}^{SK}$, take $\beta < \alpha$ such that $\varphi \in T_{\beta}^{SK}$. By Observation \ref{compord}(b), take $s' \trias s$ such that $Ord(s') = \beta$. Then by the inductive assumption $(N, \twk) \models \tau(\psi(s', \varphi))$ and therefore by Lemma \ref{monoton} $(N, \twk) \models \tau(\psi(s, \varphi))$. For the opposite implication, assuming that $(N, \twk) \models \tau(\psi(s, \varphi))$, we consider possible forms of $\varphi$.\footnote{Namely, we may assume that $\varphi$ has one of the forms $T(t)$, $\neg T(t)$, $\neg \neg \chi$, $\chi \circ \xi$, $\neg (\chi \circ \xi )$, $Qv \chi(v)$ or $\neg Qv \chi(v)$.} The proof then proceeds by cases. Since the argument in each case is very similar, we restrict ourselves to giving one example, leaving the rest for the reader to verify.

Thus, assume that $\varphi = \neg T(t)$. Then $(N, \twk) \models \neg \sentt(val(t)) \vee \exists s' \tria s~\tau(\psi (s', \neg val(t)))$. If the first disjunct holds, it follows immediately that $\varphi \in T_{\alpha}^{SK}$. If the second, choose an $s'$ with the indicated property and let $Ord(s') = \beta$. Since $\beta < \alpha$, by the inductive assumption we obtain: $\neg val(t) \in T_{\beta}^{SK}$ and since $\alpha$ is a limit ordinal, we can take an ordinal $\gamma$ and an $s'' \tria s$ such that $Ord(s'') = \gamma$, $\gamma > \beta$ and $\gamma < \alpha$. Then $\neg T(t) \in T_{\gamma}^{SK}$ and thus $\neg T(t) \in T_{\alpha}^{SK}$.
\end{proof}

\begin{corollary}\label{taulfp} If $\tau(x)$ is an arbitrary $KF$ truth predicate in $(N, \twk)$, then there is a formula $\tau'(x) \in \lt$ which defines in $(N, \twk)$ the set of sentences determined as true in the least fixed point model of $KF$.
\end{corollary}
\begin{proof}
Define $\tau'(x)$ as `$\exists s \in D~\tau(\psi(s, x))$'. The result follows immediately from Lemma \ref{leastfp} and the observation that the ordinals of elements of $D^{WK}$ are arbitrarily large below $\omega_1^{CK}$.
\end{proof}

For the proof of Lemma \ref{main2}, we assume that we have at our disposal a  $KF$ truth predicate. Our task is to show the existence of a $KF$ truth predicate which can `recognize' Weak Kleene groundedness in the sense of conditions (a) and (b) of Lemma \ref{main1}. For this we need to define a formula $G(x)$ expressing groundedness. We introduce such a formula below.

\begin{definition}\label{theta} Let $\varphi(x, y)$ be defined as $\exists z \tria x \exists s \big( s = sub(y, \ulcorner x \urcorner, z) \wedge T(s) \big)$. Define $\theta(x)$ as the diagonal formula for $\varphi(x, y)$ (for the details of the construction, see Definition \ref{diag}).
\end{definition}
In particular, $PAT$ proves that:
\begin{center}
$\theta(x) \equiv \exists z \tria x~T(\theta(z))$.
\end{center}
Now, our $G(x)$ is defined as $\neg \theta(x)$. Let us start with the following observation. 

\begin{observation}\label{thetagr} For every $\psi \in \sentt$, if $\psi$ is WK-grounded, then $(N, \tsk) \models T(\neg \theta(\psi))$.
\end{observation}
\begin{proof} The proof proceeds by ordinal induction. Assume that for every $\beta < \alpha$ and for every $\psi \in \sentt$, if $\psi$ is WK-grounded and $Ord(\psi) = \beta$, then $(N, \tsk) \models T(\neg \theta(\psi))$. Let $\psi$ be WK-grounded such that $Ord(\psi)= \alpha$. We claim that $(N, \tsk) \models T(\neg \theta(\psi))$.

By Lemma \ref{diagkf}(b), it is enough to demonstrate that $(N, \tsk) \models T(\neg \exists z \tria \psi~T(\theta(z)))$, which in turn is equivalent in $(N, \tsk)$ to `$\forall z \big( z \tria \psi \rightarrow T(\neg \theta(z)) \big)$'.

Fixing $z \tria \psi$, we notice that $z$ is WK-grounded and $Ord(z) <  \alpha$. Therefore by the inductive assumption  $(N, \tsk) \models T(\neg \theta(z))$ and thus the proof is finished.
\end{proof}

At this moment we see that already the least fixed-point model of $KF$ is able to recognize Weak Kleene groundedness, but only in a restricted sense: if the sentence $\psi$ is WK-grounded, then the model will classify $G(\psi)$ - that is, $\neg \theta(\psi)$ - as determinately true. However, in case of $\psi$ being WK-ungrounded, the least fixed-point model of $KF$ is not able to recognize $G(\psi)$ as determinately false. In effect, the condition (b) from Lemma \ref{main1} is still not satisfied. Below we are going to show how to remedy this defect. Let us start with the following definition.

\begin{definition}\label{thetahier} \hfil
\begin{itemize}
\item $T^{\theta}_0 = Th(N) \cup \{\theta(t): t \in Tm^c \wedge val(t)$ is WK-ungrounded$\}$,
\item $T^{\theta}_{\alpha +1} = J^{SK}(T^{\theta}_{\alpha}) \cup \{\theta(t): t \in Tm^c \wedge val(t)$ is WK-ungrounded$\}$,
\item $T^{\theta}_{\lambda} = \underset{\alpha < \lambda}{\bigcup} T^{\theta}_{\alpha}$,
\item $T^{\theta}$ is $T^{\theta}_{\kappa}$ for the least $\kappa$ such that $T^{\theta}_{\kappa} = T^{\theta}_{\kappa + 1}$.
\end{itemize}
\end{definition}

The definition closely resembles the usual construction of the Strong Kleene fixed-point model. The only difference is that we explicitly add the information that for all WK-ungrounded sentences $\psi$, $\theta(\psi)$ (in other words, $\neg G(\psi)$) will belong to the interpretation of the truth predicate. Note that for technical reasons, the set $\{\theta(t): t \in Tm^c \wedge val(t)$ is WK-ungrounded$\}$ is added in the definition also at successor levels. This is done in order to guarantee the full monotonicity of the construction.\footnote{Defining $T^{\theta}_{\alpha +1}$ as $J^{SK}(T^{\theta}_{\alpha})$ would bring a minor technical complication: given that $\theta(t) \in T^{\theta}_0$, why should it belong also to $T^{\theta}_1$? As defined earlier (cf. Definitions \ref{theta} and \ref{diag}), $\theta(t)$ is the formula:\vspace{2pt}
\centerline{$\exists ab [a = name(m) \wedge b = sub(m, \ulcorner y \urcorner, a) \wedge \exists z \tria t \exists s \big( s = sub(b, \ulcorner x \urcorner, z) \wedge T(s) \big)]$.}\vspace{2pt}
Clearly, already $T^{\theta}_0$ will contain witnessing statements (with fixed $a$, $b$, $z$ and $s$) for the arithmetical part of this existential formula. In other words, in  $T^{\theta}_0$ we will have (for fixed $a$, $b$, $z$ and $s$) $a = name(m) \wedge b = sub(m, \ulcorner y \urcorner, a) \wedge z \tria t \wedge s = sub(b, \ulcorner x \urcorner, z)$, with $s$ also belonging to $T^{\theta}_0$ and identical to $\theta(z)$. However, this in itself is not enough to guarantee that the Strong Kleene jump applied to $T^{\theta}_0$ will produce $\theta(t)$.}

The properties of $T^{\theta}$ are encapsulated in the following observation.

\begin{observation}\label{thetakf} $(N, T^{\theta}) \models KF$ and for every $\psi \in \sentt$:
\begin{itemize}
\item[(a)] $\psi$ is WK-grounded iff $(N, T^{\theta}) \models T(\neg \theta(\psi))$,
\item[(b)] $\psi$ is WK-ungrounded iff $(N, T^{\theta}) \models T(\theta(\psi))$.
\end{itemize}
\end{observation} 

In view of Observation \ref{thetakf}, in the proof of Lemma \ref{main2} it is enough to demonstrate that if $(N, \twk)$ defines a truth predicate of $KF$, then it defines also $T^{\theta}$. \\

\noindent
\textbf{Proof of Lemma \ref{main2}} Assume that $(N, \twk)$ defines a truth predicate of $KF$. By Corollary \ref{taulfp}, let $\tau(x)$ be a formula defining in $(N, \twk)$ the set of sentences determined as true in the least fixed-point model of $KF$.

Define:

\begin{center}
$\tau^{Compl}(x):= \neg \tau(\neg x)$.
\end{center}

Then $\tau^{Compl}(x)$ defines in $(N, \twk)$ the set of sentences determined as true in the largest fixed-point model of $KF$.\footnote{This has been observed in a more general form by \cite{Can89}. Let {\scshape Con} be the statement `$\forall \psi \neg \big(T(\psi) \wedge T(\neg \psi) \big)$' and let {\scshape Compl} be the statement `$\forall \psi \big(T(\psi) \vee T(\neg \psi) \big)$'. The general observation is that given a model $(N, T)$ of $KF +${\scshape Con}, a model $(N, T')$ of $KF +${\scshape Compl} can be obtained by defining $T'$ as the set of those sentences of $\lt$ whose negations do not belong to $T$.} Observe that in the largest fixed-point model $(N, T^{Compl})$ of $KF$, all sentences $\neg \theta(t)$ for $val(t)$ being WK-grounded will belong to $T^{Compl}$ (this is because by Observation \ref{thetagr} they belong already to $\tsk$); moreover, in such cases by Fact \ref{basic}(iii) $\theta(t)$ does not belong to $T^{Compl}$. On the other hand, if $val(t)$ is not WK-grounded, then both $\theta(t)$ and $\neg \theta(t)$ belongs to $T^{Compl}$.

Copying the idea from Definition \ref{psi}, let $\psi_1$ be the diagonal formula satisfying the following condition:
\begin{align}
	\psi_1(s, x) \equiv &\ \ x \in \sentpa \wedge T(x)\nonumber\\
	\vee&\ \  x = \ulcorner \theta(t) \urcorner \wedge T(\theta(t))\nonumber\\
	\vee&\ \  x = \ulcorner T(t) \urcorner \wedge \exists s' \vartriangleleft s T(\psi_1(s', val(t))) \nonumber\\
	\vee&\ \  x = \ulcorner \neg T(t) \urcorner \wedge \big( \exists s' \vartriangleleft s T(\psi_1 (s', \neg val(t))) \vee \neg \sentt(val(t)) \big) \nonumber\\
	\vee&\ \  x = \ulcorner \neg \neg \varphi \urcorner \wedge  \exists s' \vartriangleleft s T ( \psi_1 (s', \varphi))  \nonumber\\
	\vee&\ \  x = \ulcorner \varphi \circ \chi \urcorner \wedge  \exists s's'' \vartriangleleft s \big( T ( \psi_1 (s', \varphi)) \circ T ( \psi_1 (s'', \chi)) \big) \nonumber\\
		\vee&\ \  x = \ulcorner \neg (\varphi \circ \chi ) \urcorner \wedge  \exists s's'' \vartriangleleft s \big( T ( \psi_1 (s', \neg \varphi)) \circ_d T ( \psi_1 (s'', \neg \chi)) \big) \nonumber\\
				\vee&\ \  x = \ulcorner Qv \varphi  \urcorner \wedge  \exists s' \vartriangleleft s Qa T ( \psi_1 (s', \varphi(a))) \nonumber\\
\vee&\ \  x = \ulcorner \neg Qv \varphi  \urcorner \wedge  \exists s' \vartriangleleft s Q_d a T ( \psi_1 (s', \neg \varphi(a))) \nonumber\
\end{align}
Finally, we define:
\begin{center}
$\tau^{\theta}(x):= \exists s \big(D(s) \wedge \tau^{Compl}(\psi_1(s, x)) \big)$.
\end{center}
Arguing exactly as in the proof of Lemma \ref{leastfp}, we prove that $\tau^{\theta}(x)$ defines $T^{\theta}$ in $(N, \twk)$. $\Box$

\bibliography{ref}

\end{document}